\documentclass[11pt]{article}
\usepackage{hyperref, amsthm, amssymb, amsmath, extarrows}

\title{Rainbow fractional matchings}
\author{
  Ron Aharoni\thanks{Department of Mathematics, Technion -- Israel Institute of Technology, Technion City, Haifa 3200003, Israel}\footnotemark[1]${\ }^{,}$\footnote{Supported in part by the United States--Israel Binational Science Foundation (BSF) grant no.\ 2006099, the Israel Science Foundation (ISF) grant no. 2023464 and the Discount Bank Chair at the Technion. Email: {\tt ra@tx.technion.ac.il}.}
  \and Ron Holzman\footnotemark[1]${\ }^{,}$\thanks{Corresponding author. Supported in part by ISF grant no.\ 409/16. Email: {\tt holzman@technion.ac.il}.}
  \and Zilin Jiang\thanks{Department of Mathematics, Massachusetts Institute of Technology, Cambridge, MA 02139, USA. The work was done when Z.~Jiang was a postdoctoral fellow at Technion -- Israel Institute of Technology, and was supported in part by ISF grant nos.\ 409/16, 936/16. Email: {\tt zilinj@mit.edu}.}
}
\date{}

\newtheorem{theorem}{Theorem}[section]

\newtheorem{proposition}[theorem]{Proposition}
\theoremstyle{remark}
\newtheorem{claim}{Claim}

\begin{document}

\maketitle

\begin{abstract}
  We prove that any family $E_1, \ldots , E_{\lceil rn \rceil}$ of (not necessarily distinct) sets of edges in an $r$-uniform hypergraph, each having a fractional matching of size $n$, has a rainbow fractional matching of size $n$ (that is, a set of edges from distinct $E_i$'s which supports such a fractional matching). When the hypergraph is $r$-partite and $n$ is an integer, the number of sets needed goes down from $rn$ to $rn-r+1$. The problem solved here is a fractional version of the corresponding problem about rainbow matchings, which was solved by Drisko and by Aharoni and Berger in the case of bipartite graphs, but is open for general graphs as well as for $r$-partite hypergraphs with $r>2$. Our topological proof is based on a result of Kalai and Meshulam about a simplicial complex and a matroid on the same vertex set.
\end{abstract}



\section{Introduction}

Given a family $M_1, \ldots , M_m$ of (not necessarily distinct) matchings in a graph $G$, a \emph{rainbow matching} of size $n$ is a matching $\{e_1, \ldots , e_n\}$ with $e_i \in M_{\sigma(i)}$, $i=1, \ldots , n$, so that $\sigma(1), \ldots , \sigma(n)$ are distinct. Drisko proved the following theorem (which he stated using Latin-rectangles terminology).

\begin{theorem}[Theorem~1 of Drisko~\cite{Dr}] \label{drisko}
  Let $G=K_{n,k}$ with $n \le k$. Any family of $2n-1$ matchings of size $n$ in $G$ has a rainbow matching of size $n$.
\end{theorem}

Drisko applied his theorem to questions about complete mappings for group actions, and difference sets in groups. Later, Alon \cite{Al} pointed out connections to additive number theory, and showed in particular that Theorem~\ref{drisko} implies the well-known result of Erd\H{o}s, Ginzburg and Ziv \cite{EGZ}.

Aharoni and Berger re-formulated and re-proved Drisko's theorem, while removing the assumption that one side of the bipartite graph is of size $n$. Namely, they established the following.

\begin{theorem}[Theorem~4.1 of Aharoni and Berger~\cite{AB}] \label{ab}
  Any family of $2n-1$ matchings of size $n$ in a bipartite graph has a rainbow matching of size $n$.
\end{theorem}

Drisko showed that the parameter $2n-1$ is best possible: Consider a cycle of length $2n$, and a family of $2n-2$ matchings consisting of $n-1$ copies of each of its two perfect matchings. He conjectured that this is the only extremal example, and Aharoni, Kotlar and Ziv~\cite{AKZ} proved this not only in Drisko's $K_{n,k}$ setting but in general bipartite graphs.

If the bipartiteness assumption is removed, the result of Theorem~\ref{ab} is no longer true. Indeed, Bar\'at, Gy\'arf\'as and S\'ark\"ozy \cite{BGS} showed that for even $n$, one can start as above with $n-1$ copies of each of the two perfect matchings in $C_{2n}$ (whose vertices we denote by $1, \ldots , 2n$ in cyclic order), and add one extra matching $\{1\,3, 2\,4, 5\,7, 6\,8, \ldots , (2n-3)\,(2n-1), (2n-2)\,2n\}$, still without having a rainbow matching of size $n$.

No examples with more than $2n-1$ matchings of size $n$ in arbitrary graphs are known which have no rainbow matching of size $n$. It may well be the case that $2n$ matchings suffice, perhaps already $2n-1$ are enough if $n$ is odd. However, the best known result for arbitrary graphs is the following one due to Aharoni et al.

\begin{theorem}[Theorem~1.9 of Aharoni et al.~\cite{ABCHS}] \label{abchs}
Any family of $3n-2$ matchings of size $n$ in an arbitrary graph has a rainbow matching of size $n$.
\end{theorem}

In the absence of a tight result for matchings in arbitrary graphs, we are led to consider the fractional version of the problem. Recall that a \emph{fractional matching} for a set $E$ of edges is a function $f\colon E \to \mathbb{R}_+$ such that $\sum_{e \ni v} f(e) \le 1$ for every vertex $v$. The size of $f$ is $\sum_{e \in E} f(e)$. The fractional matching number $\nu^*(E)$ is the maximal size of a fractional matching for $E$. In bipartite graphs, by K\"onig's theorem, this number is equal to the matching number $\nu(E)$. In an arbitrary graph, $\nu^*(E)$ may be larger than $\nu(E)$, and is either an integer or a half-integer.

Given a family $E_1, \ldots , E_m$ of (not necessarily distinct) sets of edges in a graph $G$, a \emph{rainbow fractional matching} of size $n$ is a set of edges $\{e_1, \ldots , e_k\}$ with $e_i \in E_{\sigma(i)}$, $i=1, \ldots , k$, so that $\sigma(1), \ldots , \sigma(k)$ are distinct, together with a fractional matching $f\colon \{e_1, \ldots , e_k\} \to \mathbb{R}_+$ of size $n$.

Due to K\"onig's theorem, the following is equivalent to Theorem~\ref{ab}.

\begin{theorem} \label{fracbip}
  Let $n$ be a positive integer. Any family $E_1, \ldots , E_{2n-1}$ of sets of edges in a bipartite graph, such that $\nu^*(E_i) \ge n$ for $i=1, \ldots , 2n-1$, has a rainbow fractional matching of size $n$.
\end{theorem}

We get here a new proof of Theorem~\ref{ab}/\ref{fracbip} by considering fractional matchings. Our unified approach also yields the following new result for arbitrary graphs.

\begin{theorem} \label{graphs}
  Let $n$ be a positive integer or half-integer. Any family $E_1, \ldots , E_{2n}$ of sets of edges in an arbitrary graph, such that $\nu^*(E_i) \ge n$ for $i=1, \ldots , 2n$, has a rainbow fractional matching of size~$n$.
\end{theorem}

Thus, the cost of allowing arbitrary graphs instead of just bipartite ones is only one more family in the fractional case. This suggests that the difficulty of the problem for (integral) matchings in arbitrary graphs has to do with the fact that matchings behave differently than fractional matchings in such graphs.

The parameter $2n$ is best possible for fractional matchings in arbitrary graphs. Indeed, if $n$ is a half-integer, say $n=k+\frac{1}{2}$, $k \ge 1$, take $E_1, \ldots , E_{2k}$ to be $2k$ copies of the edge set of a cycle of length $2k+1$. Then $\nu^*(E_i)=n$, but we need all $2k+1$ edges to achieve this, so there is no rainbow fractional matching of size $n$. If $n \ge 3$ is an integer, we can repeat the above construction using two vertex-disjoint odd cycles whose lengths add up to $2n$.

In fact, our approach is more general, as we consider $r$-uniform hypergraphs for any $r \ge 2$. The definition of a fractional matching for a set $E$ of edges of size $r$ is the same as above, but for $r>2$ the number $\nu^*(E)$ may have as its fractional part any rational number in $[0,1)$.

Our main result, of which Theorem~\ref{graphs} is the case $r=2$, is the following.

\begin{theorem} \label{hyper}
Let $r \ge 2$ be an integer, and let $n$ be a positive rational number. Any family $E_1, \ldots , E_{\lceil rn \rceil}$ of sets of edges in an $r$-uniform hypergraph, such that $\nu^*(E_i) \ge n$ for $i=1, \ldots , \lceil rn \rceil$, has a rainbow fractional matching of size $n$.
\end{theorem}

Constructions similar to those described above for $r=2$ show that when $rn$ is an integer, we cannot do with fewer than $rn$ sets of edges. When $rn$ is not an integer, we believe that $\lfloor rn \rfloor$ sets (instead of $\lceil rn \rceil$) suffice, but our method of proof is not capable of showing this.

Just like bipartite graphs behave slightly better than general graphs with respect to guaranteeing a rainbow fractional matching, so do $r$-partite hypergraphs compared to general $r$-uniform hypergraphs. Recall that a hypergraph is $r$-\emph{partite} if there exists a partition $A_1, \ldots , A_r$ of the vertex set, so that every edge consists of exactly one vertex from each $A_i$.

The following is a sharpening of the main result when confined to $r$-partite hypergraphs and integer values of $n$.

\begin{theorem} \label{partite}
Let $r \ge 2$ and $n \ge 1$ be integers. Any family $E_1, \ldots , E_{rn-r+1}$ of sets of edges in an $r$-partite hypergraph, such that $\nu^*(E_i) \ge n$ for $i=1, \ldots , rn-r+1$, has a rainbow fractional matching of size $n$.
\end{theorem}

While the requirement that $n$ be an integer is restrictive, we note that it always holds when the hypergraph is equi-partite and each $E_i$ has a perfect fractional matching. In this case $n=|A_1|= \cdots =|A_r|$ is the size of a perfect fractional matching, and Theorem~\ref{partite} guarantees the existence of a rainbow perfect fractional matching. The parameter $rn-r+1$ is best possible here. Indeed, taking $n=r-1$ to be a prime power, and letting $E_i$ be the set of lines of a truncated projective plane of order $n$, we have $|E_i|=n^2=rn-r+1$ and $\nu^*(E_i)=n$. However, we need all edges to achieve this value of $\nu^*$ (see Theorem~2.1 of F\"uredi \cite{Fu}), so fewer copies of $E_i$ do not suffice for a rainbow perfect fractional matching. We also observe that Theorem~\ref{partite} specializes to Theorem~\ref{fracbip} in the case $r=2$.

Thus, we only have to prove Theorems~\ref{hyper} and~\ref{partite}. This will be done in Section~\ref{proof}. The proof is based on a topological result, due to Kalai and Meshulam \cite{KM}, which they developed as an extension of the colorful versions, due to Lov\'asz and B\'ar\'any \cite{Ba}, of the Helly and Carath\'eodory theorems. The necessary notions will be reviewed in Section~\ref{tools}.

In the case of perfect fractional matchings, both Theorem~\ref{hyper} and Theorem~\ref{partite} follow from B\'ar\'any's colorful Carath\'eodory theorem.

\begin{theorem}[Theorem~2.2 of B\'ar\'any~\cite{Ba}]
  Suppose $V_1, \ldots, V_n\subseteq\mathbb{R}^n$ and $a\in \mathrm{pos}\,V_i$ (the conical hull of $V_i$) for $i = 1, \ldots, n$. Then there exist elements $v_i \in V_i$ for each $i$ such that $a \in \mathrm{pos}\{v_1,  \ldots, v_n\}$.
\end{theorem}

To derive Theorem~\ref{hyper} in the case of $rn$ vertices, note that a set $E$ of edges has a perfect fractional matching if and only if $\vec{1} \in \mathrm{pos}\{\chi_e:\,e \in E\}$, where $\chi_e$ is the indicating vector of $e$ as a subset of the vertex set. To derive the corresponding case of Theorem~\ref{partite} note also that since the sum of entries of $\chi_e$ in every side of the hypergraph is the same, the vectors $\chi_e$ all live in an $(rn-r+1)$-dimensional subspace.

We end the introduction with a comment on the relation between the question of the existence of a rainbow fractional matching, dealt with here, and the original question concerning the existence of a rainbow (integral) matching. For graphs, the state of affairs in the two questions is very similar: the questions are equivalent in the bipartite case, and in the general case we know that $2n$ sets suffice for a rainbow fractional matching of size $n$, while $3n-2$ suffice for a rainbow matching of size $n$ (and it may well be the case that already $2n$ suffice). For $r$-uniform hypergraphs with $r>2$, however, the two questions diverge. While $\lceil rn \rceil$ sets suffice for a rainbow fractional matching of size $n$, results of Aharoni and Berger \cite{AB} and Alon \cite{Al} show that even in $r$-partite hypergraphs, the number of matchings of size $n$ required for a rainbow matching of size $n$ grows exponentially in $r$ (the nature of the dependence on $n$ is far from being understood).

\section{Topological tools}\label{tools}

Let $V$ be a finite vertex set. A \emph{simplicial complex} on $V$ is a family $X$ of subsets of $V$ (called simplices or faces) which is downward closed (i.e., $\sigma \subseteq \tau \in X \Longrightarrow \sigma \in X$). A face $\tau \in X$ which is maximal with respect to inclusion is called a facet of $X$. A \emph{matroid} on $V$ is a simplicial complex $M$ which satisfies the augmentation property (i.e., $\sigma, \tau \in M,\,\, |\sigma| < |\tau| \Longrightarrow \sigma \cup \{v\} \in M$ for some $v \in \tau \setminus \sigma$). The rank function $\rho$ of $M$ assigns to every subset $U \subseteq V$ the number $\rho(U) = \max \{|\sigma|:\, \sigma \in M, \sigma \subseteq U\}$.

Let $X$ be a simplicial complex, and let $d$ be a positive integer. If $\sigma$ is a face which is contained in a unique facet $\tau$ of $X$, and $|\sigma| \le d$, the operation of removing from $X$ the face $\sigma$ and all faces containing it is called an \emph{elementary $d$-collapse}. A complex $X$ is $d$-\emph{collapsible} if there exists a sequence of elementary $d$-collapses $X = X_0 \xlongrightarrow[\sigma_1]{} X_1 \xlongrightarrow[\sigma_2]{} X_2 \xlongrightarrow[\sigma_3]{} \cdots \xlongrightarrow[\sigma_m]{} X_m = \{ \emptyset \}$ that reduces $X$ to $\{\emptyset\}$. Wegner \cite{We} introduced this property, and observed that every $d$-collapsible complex $X$ is $d$-Leray, i.e., all its induced subcomplexes have trivial homology in dimensions $d$ and above. As we will not work here directly with $d$-Lerayness, we do not give a detailed definition.

We will use the following observation about $d$-collapsibility (the analogous statement for $d$-Lerayness appeared as Proposition 14(i) of Alon et al. \cite{AKMM}).

\begin{proposition} \label{blow}
Let $X$ be a simplicial complex on $V$. Consider a blown-up vertex set $\tilde{V}$ which contains one or more copies of each vertex in $V$, i.e., $\tilde{V} = \{ (v,i): \, v \in V,\, i=1,\ldots,k_v \}$, and the simplicial complex $\tilde{X}$ on $\tilde{V}$ having as faces those subsets of $\tilde{V}$ whose projection on $V$ is a face of $X$. Let $d$ be a positive integer. If $X$ is $d$-collapsible then so is $\tilde{X}$.
\end{proposition}

\begin{proof} It suffices to prove the statement for the case when only one vertex $u \in V$ is blown up, being replaced by two copies of it. The general case will follow by repeated use of this special case of the proposition.

Adapting our notation to the special case, we have $\tilde{V} = V \cup \{\tilde{u}\}$, the restriction of $\tilde{X}$ to $V$ coincides with $X$, and vertices $u$, $\tilde{u}$ are clones, in the sense that for all $\sigma \subseteq V \setminus \{u\}$ we have $\sigma \cup \{u\} \in \tilde{X} \Longleftrightarrow \sigma \cup \{\tilde{u}\} \in \tilde{X} \Longleftrightarrow \sigma \cup \{u,\tilde{u}\} \in \tilde{X}$. Now, given a sequence of elementary $d$-collapses $X = X_0 \xlongrightarrow[\sigma_1]{} X_1 \xlongrightarrow[\sigma_2]{} X_2 \xlongrightarrow[\sigma_3]{} \cdots \xlongrightarrow[\sigma_m]{} X_m = \{ \emptyset \}$, we show how to modify the sequence so as to reduce $\tilde{X}$ to $\{ \emptyset \}$. Whenever the original $\sigma_i$ does not contain $u$, we use the same $\sigma_i$ in our modified sequence. Whenever the original $\sigma_i$ contains $u$, we replace the step by a double step using $\sigma_i$ and $(\sigma_i \setminus \{u\}) \cup \{\tilde{u}\}$, respectively. One can verify by induction on $i$ that these are indeed legal elementary $d$-collapses, and that after each step or double step the remaining subcomplex $\tilde{X}_i$ satisfies: its restriction to $V$ coincides with $X_i$, and vertices $u$, $\tilde{u}$ are clones with respect to $\tilde{X}_i$. It follows in particular that $\tilde{X}_m = \{ \emptyset \}$ as required.
\end{proof}

The following result of Kalai and Meshulam is our main tool.

\begin{theorem}[Theorem~1.6 of Kalai and Meshulam \cite{KM}] \label{km}
  Let $X$ be a simplicial complex and let $M$ be a matroid with rank function $\rho$, both on the same vertex set $V$, such that $M \subseteq X$. Let $d$ be a positive integer. If $X$ is $d$-Leray, then there exists a face $\tau \in X$ such that $\rho(V \setminus \tau) \le d$.
\end{theorem}

We will only use the conclusion of this theorem under the stronger assumption that $X$ is $d$-collapsible. (Alternatively, we could use Theorem~2.1 of \cite{KM} which, under this stronger assumption, derives a slightly stronger conclusion than we need.)

\section{Proofs} \label{proof}

The main ingredient of the proof of Theorem~\ref{hyper} is the following.

\begin{theorem} \label{collapsible}
  Let $r \ge 2$ be an integer, let $E$ be the set of edges of some $r$-uniform hypergraph, and let $n>1$ be a rational number. The simplicial complex $X$ on $E$, defined by $X = \{E' \subseteq E:\, \nu^*(E') < n\}$, is $(\lceil rn \rceil - 1)$-collapsible.
\end{theorem}

Note that the condition $n>1$ only serves to avoid a trivial case: when $0<n \le 1$ our complex is just $\{ \emptyset \}$. Before proving this theorem, we show how Theorem~\ref{hyper} follows from it via Proposition~\ref{blow} and Theorem~\ref{km}.

\begin{proof}[Proof of Theorem~\ref{hyper}]
  Let $E_1, \ldots , E_{\lceil rn \rceil}$ be sets of edges of size $r$ such that $\nu^*(E_i) \ge n$ for $i=1, \ldots , \lceil rn \rceil$. Assume for the sake of contradiction that there is no rainbow fractional matching of size $n$. Clearly, we must have $n>1$.

  We shall apply Theorem~\ref{km} to a simplicial complex and a matroid on the set $\tilde{E}$ consisting of all edges in $E = \bigcup_{i=1}^{\lceil rn \rceil} E_i$ labeled by the sets they appear in, i.e.,
  \[
    \tilde{E} = \left\{(e,i):\, e \in E_i\right\}.
  \]
  The simplicial complex on $\tilde{E}$ that we consider is
  \[
    \tilde{X} = \{\tilde{E}' \subseteq \tilde{E}:\, \nu^*(\{e:\, \exists i \text{ s.t. } (e,i) \in \tilde{E}'\}) < n\}.
  \]
  Note that the complex $\tilde{X}$ is obtained from the complex $X$ of Theorem~\ref{collapsible} by a blowing up construction as in Proposition~\ref{blow}. Since, by Theorem~\ref{collapsible}, $X$ is $(\lceil rn \rceil - 1)$-collapsible, it follows by Proposition~\ref{blow} that so is $\tilde{X}$.

  We consider the partition matroid $M$ on $\tilde{E}$ with parts corresponding to $E_1, \ldots , E_{\lceil rn \rceil}$, i.e.,
  \[
    M = \{\tilde{E}' \subseteq \tilde{E}:\, |\tilde{E}' \cap (E_i \times \{i\})| \le 1, \,\,i=1, \ldots , \lceil rn \rceil\}.
  \]
  Our assumption that there is no rainbow fractional matching of size $n$ means that $M \subseteq \tilde{X}$. As $\tilde{X}$ is $(\lceil rn \rceil - 1)$-collapsible, and hence ($\lceil rn \rceil - 1$)-Leray, it follows from Theorem~\ref{km} that there exists $\tilde{E}' \in \tilde{X}$ such that $\rho(\tilde{E} \setminus \tilde{E}') \le \lceil rn \rceil - 1$. The latter means that $\tilde{E} \setminus \tilde{E}'$ entirely misses one of the parts in the partition. Thus, there exists $i$ such that $E_i \times \{i\} \subseteq \tilde{E}'$, which is impossible because $\nu^*(E_i) \ge n$ and $\tilde{E}' \in \tilde{X}$.
\end{proof}

 We prove below a generalization of Theorem~\ref{collapsible} which allows for edge and vertex weights. To explain the role of these weights in facilitating the proof, we note that we were inspired by Wegner's \cite{We} proof that the nerve of every finite family of convex sets in $\mathbb{R}^d$ is $d$-collapsible. Our proof essentially follows his sliding hyperplane argument, but instead of using the intersections of the convex sets to guide the choice of collapse moves, we use the values of $\nu^*$ on subsets of $E$. Wegner's generic choice of the direction in which the hyperplane moves is achieved here by slightly perturbing the vertex weights. Where he slightly modifies the convex sets in order to apply induction, we slightly perturb the edge weights.

Let $E$ be a set of edges on a vertex set $V$. Let $\mathbf{a} = \{a_e\}_{e \in E}$ and $\mathbf{b} = \{b_v\}_{v \in V}$ be systems of positive real weights. For any $E' \subseteq E$ we consider
\begin{eqnarray*}
  \nu^*_{\mathbf{a},\mathbf{b}}(E') \,\ = & \max & \sum_{e \in E'} a_e f(e) \\
  & \text{ s.t. } & \sum_{e \ni v} f(e) \le b_v \quad \forall v \in V \\
  & & f(e) \ge 0 \quad \forall e \in E'
\end{eqnarray*}
By linear programming duality, the above value is equal to
\begin{eqnarray*}
  \tau^*_{\mathbf{a},\mathbf{b}}(E') \,\ = & \min & \sum_{v \in V} b_v g(v) \\
  & \textrm{ s.t. } & \sum_{v \in e} g(v) \ge a_e \quad \forall e \in E' \\
  & & g(v) \ge 0 \quad \forall v \in V
\end{eqnarray*}
The case where all $a_e$ and all $b_v$ are equal to $1$ gives the standard fractional matching and covering numbers.

The following extends Theorem~\ref{collapsible} to the weighted set-up.
\begin{theorem} \label{weighted}
Let $r \ge 2$ be an integer, and let $E$ be the set of edges of some $r$-uniform hypergraph on a vertex set $V$. Let $\mathbf{a} = \{a_e\}_{e \in E}$ and $\mathbf{b} = \{b_v\}_{v \in V}$ be systems of positive real weights, and write $\underline{a} = \min_{e \in E} a_e$ and $\underline{b} = \min_{v \in V} b_v$. Let $n > \underline{a} \underline{b}$ be a real number. The simplicial complex $X = X_{\mathbf{a}, \mathbf{b}, n}$ on $E$, defined by $X = \{E' \subseteq E:\, \nu^*_{\mathbf{a},\mathbf{b}}(E') < n\}$, is $(\lceil \frac{rn}{\underline{a} \underline{b}} \rceil - 1)$-collapsible.
\end{theorem}

\begin{proof}
  We proceed by induction on $|X|$. If $X=\{\emptyset\}$ there is nothing to show, so we assume that $|X|>1$.

  We will assume that for any $E' \subseteq E$ there is a unique function $g$ on $V$ that attains the minimum in the program defining $\tau^*_{\mathbf{a},\mathbf{b}}(E')$. Indeed, we can achieve this by slightly perturbing the vertex weights $\mathbf{b}$. If the perturbation is small enough and does not decrease any $b_v$, it does not affect the complex $X = X_{\mathbf{a},\mathbf{b}, n}$, nor the value of $\lceil \frac{rn}{\underline{a}\underline{b}} \rceil - 1$ (due to rounding up).

  Let $\overline{n} < n$ be defined by $\overline{n} = \max_{E' \in X} \nu^*_{\mathbf{a},\mathbf{b}}(E')$, and let $\overline{E} \in X$ be a set of edges which attains this maximum, and is inclusion-minimal among such sets. We will show that removing $\overline{E}$ and all its supersets from $X$ is an elementary ($\lceil \frac{rn}{\underline{a}\underline{b}} \rceil - 1$)-collapse, which leaves a subcomplex of $X$ to which induction may be applied. This is done in the following three claims.

  \medskip

  \begin{claim}
    $\overline{E}$ is contained in a unique facet of $X$.
  \end{claim}

  Let $E^+ = \{e \in E \setminus \overline{E}:\, \overline{E} \cup \{e\} \in X\}$. Let $e$ be any edge in $E^+$. By the maximality of $\overline{n}$, we have $\nu^*_{\mathbf{a},\mathbf{b}}(\overline{E} \cup \{e\}) = \nu^*_{\mathbf{a},\mathbf{b}}(\overline{E}) = \overline{n}$, and hence $\tau^*_{\mathbf{a},\mathbf{b}}(\overline{E} \cup \{e\}) = \tau^*_{\mathbf{a},\mathbf{b}}(\overline{E}) = \overline{n}$. By our assumption above, there is a unique function $g$ on $V$ that attains the minimum defining $\tau^*_{\mathbf{a},\mathbf{b}}(\overline{E})$, so this function must also satisfy the constraint $\sum_{v \in e} g(v) \ge a_e$ for the extra edge $e$. As this is true for each $e \in E^+$, the same function $g$ satisfies the constraints for all edges in $\overline{E} \cup E^+$, and therefore $\tau^*_{\mathbf{a},\mathbf{b}}(\overline{E} \cup E^+) = \overline{n}$ as well, implying that $\overline{E} \cup E^+ \in X$. Thus, $\overline{E} \cup E^+$ is the unique facet of $X$ that contains $\overline{E}$.

  \medskip

  \begin{claim} \label{claim2}
    $|\overline{E}| \le \lceil \frac{rn}{\underline{a}\underline{b}} \rceil - 1$.
  \end{claim}

  Consider the space $\mathbb{R}^{\overline{E}}$ of real-valued functions $f$ defined on $\overline{E}$. The constraints $\sum_{e \ni v} f(e) \le b_v$ for $v \in V$ and $f(e) \ge 0$ for $e \in \overline{E}$ define a polytope $P$ in $\mathbb{R}^{\overline{E}}$, and the maximum of $\sum_{e \in \overline{E}} a_e f(e)$ over $P$ equals $\overline{n}$. Hence there exists a vertex $\overline{f}$ of $P$ at which the maximum is attained, i.e., $\sum_{e \in \overline{E}} a_e \overline{f}(e) = \overline{n}$. Any vertex of $P$ must satisfy at least $|\overline{E}|$ of the constraints defining $P$ as equalities. However, if $\overline{f}(e)=0$ for some $e \in \overline{E}$, then $\nu^*_{\mathbf{a},\mathbf{b}}(\overline{E} \setminus \{e\}) = \nu^*_{\mathbf{a},\mathbf{b}}(\overline{E}) = \overline{n}$, contradicting the choice of $\overline{E}$ as inclusion-minimal. Therefore, we must have $\sum_{e \ni v} \overline{f}(e) = b_v$ for all $v \in U$, where $U$ is some subset of $V$ of size $|\overline{E}|$. Now \[|U|\underline{b} \le \sum_{v \in U} b_v = \sum_{v \in U} \sum_{e \ni v} \overline{f}(e) = \sum_{e \in \overline{E}} \sum_{v \in e \cap U} \overline{f}(e) \le r \sum_{e \in \overline{E}} \overline{f}(e) \le \frac{r}{\underline{a}} \sum_{e \in \overline{E}} a_e \overline{f}(e) = \frac{r \overline{n}}{\underline{a}}.\]
Thus $|\overline{E}| = |U| \le \frac{r \overline{n}}{\underline{a}\underline{b}} < \frac{rn}{\underline{a}\underline{b}}$, and since $|\overline{E}|$ is an integer it is at most $\lceil \frac{rn}{\underline{a}\underline{b}} \rceil - 1$.

  \medskip

  \begin{claim}
    Let $\widehat{X} = \{E' \in X:\, E' \nsupseteq \overline{E}\}$ be the remaining subcomplex of $X$. Then either $\widehat{X} = \{\emptyset\}$, or there exists a system of positive real edge weights $\widehat{\mathbf{a}} = \{\widehat{a}_e\}_{e \in E}$ so that $\widehat{X}$ is the complex corresponding to $\widehat{\mathbf{a}}, \mathbf{b}, \overline{n}$, i.e., \[\widehat{X} = \{E' \subseteq E:\, \nu^*_{\widehat{\mathbf{a}},\mathbf{b}}(E') < \overline{n}\},\] and the inequalities $\overline{n} > \widehat{\underline{a}} \underline{b}$, $\frac{r \overline{n}}{\widehat{\underline{a}}\underline{b}} \le \frac{rn}{\underline{a}\underline{b}}$ hold.
  \end{claim}

  We are going to show that for sufficiently small $\varepsilon > 0$, the edge weights $\widehat{\mathbf{a}} = \{ \widehat{a}_e\}_{e \in E}$ defined by
  \[
    \widehat{a}_e = \begin{cases}
      a_e & \text{if } e \in \overline{E} \\
      a_e - \varepsilon & \text{if } e \notin \overline{E}
    \end{cases}
  \]
  satisfy the requirements of the claim. To show that $\widehat{X} = \{E' \subseteq E:\, \nu^*_{\widehat{\mathbf{a}},\mathbf{b}}(E') < \overline{n}\}$, we consider three kinds of subsets $E' \subseteq E$. If $E' \notin X$ (and so $E' \notin \widehat{X}$), we know that $\nu^*_{\mathbf{a},\mathbf{b}}(E') \ge n$, and therefore $\nu^*_{\widehat{\mathbf{a}},\mathbf{b}}(E') > \overline{n}$ for small enough $\varepsilon$, since $\overline{n} < n$. If $E' \in X$ but $E' \supseteq \overline{E}$ (and so $E' \notin \widehat{X}$), we have $\nu^*_{\widehat{\mathbf{a}},\mathbf{b}}(E') \ge \nu^*_{\widehat{\mathbf{a}},\mathbf{b}}(\overline{E}) = \nu^*_{\mathbf{a},\mathbf{b}}(\overline{E}) = \overline{n}$. For the third kind, suppose that $E' \in X$ and $E' \nsupseteq \overline{E}$ (and so $E' \in \widehat{X}$), and assume for the sake of contradiction that $\nu^*_{\widehat{\mathbf{a}},\mathbf{b}}(E') \ge \overline{n}$. Let $f$ be a function on $E'$ which satisfies the constraints of the program defining $\nu^*_{\widehat{\mathbf{a}},\mathbf{b}}(E')$ and gives $\sum_{e \in E'} \widehat{a}_e f(e) \ge \overline{n}$. The support of $f$ cannot be contained in $E' \cap \overline{E}$, because the latter is a proper subset of $\overline{E}$, and this would contradict the choice of $\overline{E}$ as inclusion-minimal. Hence there exists an edge $e \in E' \setminus \overline{E}$ with $f(e) > 0$, and therefore $\sum_{e \in E'} a_e f(e) > \sum_{e \in E'} \widehat{a}_e f(e) \ge \overline{n}$, contradicting the maximality of $\overline{n}$. This completes the proof that $\widehat{X} = \{E' \subseteq E:\, \nu^*_{\widehat{\mathbf{a}},\mathbf{b}}(E') < \overline{n}\}$. Since we may assume that $\widehat{X} \ne \{\emptyset\}$, and clearly $\nu^*_{\widehat{\mathbf{a}},\mathbf{b}}(\{e\}) \ge \widehat{\underline{a}}\underline{b}$ for any single edge, it follows that $\overline{n} > \widehat{\underline{a}}\underline{b}$. Finally, since $\overline{n} < n$, choosing $\varepsilon$ small enough when defining $\widehat{\mathbf{a}}$ guarantees that $\frac{r \overline{n}}{\widehat{\underline{a}}\underline{b}} \le \frac{rn}{\underline{a}\underline{b}}$ holds.

  \medskip

  Applying the induction hypothesis to $\widehat{X}$ completes the proof of Theorem~\ref{weighted}.
\end{proof}

Having proved Theorem~\ref{hyper}, we now indicate how to get the improvement stated in Theorem~\ref{partite} for the $r$-partite case with integer $n$.

\begin{proof}[Proof of Theorem~\ref{partite}]
  The proof follows the same line, except that in the $r$-partite case the conclusion of Theorem~\ref{weighted} becomes: $X$ is $(r \lfloor \frac{\overline{n}}{\underline{a}\underline{b}} \rfloor)$-collapsible, where $\overline{n} = \max_{E' \in X} \nu^*_{\mathbf{a},\mathbf{b}}(E')$ as defined in the original proof.

  To establish the corresponding bound in Claim~\ref{claim2}, we decompose the set of vertices $U$ for which $\sum_{e \ni v} \overline{f}(e) = b_v$ holds, as $U = \bigcup_{i=1}^r U_i$, where $U_i$ is the intersection of $U$ with the $i$-th part of the vertex set. Then we can write for each $i$:\[|U_i|\underline{b} \le \sum_{v \in U_i} b_v = \sum_{v \in U_i} \sum_{e \ni v} \overline{f}(e) = \sum_{e \in \overline{E}} \sum_{v \in e \cap U_i} \overline{f}(e) \le \sum_{e \in \overline{E}} \overline{f}(e) \le \frac{1}{\underline{a}} \sum_{e \in \overline{E}} a_e \overline{f}(e) = \frac{\overline{n}}{\underline{a}}.\] Thus $|U_i| \le \lfloor \frac{\overline{n}}{\underline{a}\underline{b}} \rfloor$, and summing these inequalities for $i=1, \ldots , r$ we obtain $|\overline{E}| = |U| \le r \lfloor \frac{\overline{n}}{\underline{a}\underline{b}} \rfloor$.

  Just as in the original proof, we apply the induction hypothesis to the remaining subcomplex $\widehat{X}$, getting that $\widehat{X}$ is $(r \lfloor \frac{\widehat{\overline{n}}}{\widehat{\underline{a}}\underline{b}} \rfloor)$-collapsible, where $\widehat{\overline{n}} = \max_{E' \in \widehat{X}} \nu^*_{\widehat{\mathbf{a}},\mathbf{b}}(E')$. We have $\widehat{\overline{n}} < \overline{n}$ and therefore, by choosing $\varepsilon$ small enough when defining $\widehat{\mathbf{a}}$, we guarantee that $r \lfloor \frac{\widehat{\overline{n}}}{\widehat{\underline{a}}\underline{b}} \rfloor \le r \lfloor \frac{\overline{n}}{\underline{a}\underline{b}} \rfloor$, so the induction goes through.

  In the unweighted case, the result just proved shows that $X$ is $(r \lfloor \overline{n} \rfloor)$-collapsible. We assume in Theorem~\ref{partite} that $n$ is an integer, and clearly we may assume $n>1$. As $\overline{n} < n$ we have $\lfloor \overline{n} \rfloor \le n-1$, thus $X$ is $(rn-r)$-collapsible. Just as in the proof of Theorem~\ref{hyper}, this implies the conclusion of Theorem~\ref{partite} with $rn-r+1$ sets of edges.
\end{proof}

\section*{Acknowledgements}

We are grateful to Dani Kotlar, Roy Meshulam and Ran Ziv for helpful discussions.

\bibliographystyle{plain}
\bibliography{rainbow_fractional_matchings}

\begin{thebibliography}{10}

\bibitem{AB}
Ron Aharoni and Eli Berger.
\newblock Rainbow matchings in {$r$}-partite {$r$}-graphs.
\newblock {\em Electron. J. Combin.}, 16(1):Research Paper 119, 9, 2009.

\bibitem{ABCHS}
Ron Aharoni, Eli Berger, Maria Chudnovsky, David Howard, and Paul Seymour.
\newblock Large rainbow matchings in general graphs.
\newblock {\em European J. Combin.}, 79:222--227, 2019.
\newblock \href{http://arxiv.org/abs/1611.03648}{\tt
  arXiv:1611.03648[math.CO]}.

\bibitem{AKZ}
Ron Aharoni, Dani Kotlar, and Ran Ziv.
\newblock Uniqueness of the extreme cases in theorems of {D}risko and
  {E}rd{\H{o}}s--{G}inzburg--{Z}iv.
\newblock {\em European J. Combin.}, 67:222--229, 2018.
\newblock \href{http://arxiv.org/abs/1511.05775}{\tt
  arXiv:1511.05775[math.CO]}.

\bibitem{Al}
Noga Alon.
\newblock Multicolored matchings in hypergraphs.
\newblock {\em Mosc. J. Comb. Number Theory}, 1(1):3--10, 2011.

\bibitem{AKMM}
Noga Alon, Gil Kalai, Ji\v{r}\'{\i} Matou\v{s}ek, and Roy Meshulam.
\newblock Transversal numbers for hypergraphs arising in geometry.
\newblock {\em Adv. in Appl. Math.}, 29(1):79--101, 2002.

\bibitem{Ba}
Imre B{\'a}r\'any.
\newblock A generalization of {C}arath\'eodory's theorem.
\newblock {\em Discrete Math.}, 40(2-3):141--152, 1982.

\bibitem{BGS}
J\'anos Bar\'at, Andr\'as Gy\'arf\'as, and G\'abor~N. S\'ark\"ozy.
\newblock Rainbow matchings in bipartite multigraphs.
\newblock {\em Period. Math. Hungar.}, 74(1):108--111, 2017.
\newblock \href{http://arxiv.org/abs/1505.01779}{\tt
  arXiv:1505.01779[math.CO]}.

\bibitem{Dr}
Arthur~A. Drisko.
\newblock Transversals in row-{L}atin rectangles.
\newblock {\em J. Combin. Theory Ser. A}, 84(2):181--195, 1998.

\bibitem{EGZ}
P.~Erd\H{o}s, A.~Ginzburg, and A.~Ziv.
\newblock Theorem in the additive number theory.
\newblock {\em Bull. Res. Counc. Israel Sect. F Math. Phys.}, 10F(1):41--43,
  1961.

\bibitem{Fu}
Zolt\'an F{\"u}redi.
\newblock Covering the complete graph by partitions.
\newblock {\em Discrete Math.}, 75(1-3):217--226, 1989.
\newblock Graph theory and combinatorics (Cambridge, 1988).

\bibitem{KM}
Gil Kalai and Roy Meshulam.
\newblock A topological colorful {H}elly theorem.
\newblock {\em Adv. Math.}, 191(2):305--311, 2005.

\bibitem{We}
Gerd Wegner.
\newblock {$d$}-collapsing and nerves of families of convex sets.
\newblock {\em Arch. Math. (Basel)}, 26:317--321, 1975.

\end{thebibliography}

\end{document}